\documentclass[12pt]{article}
\usepackage[dvips]{color}
\usepackage{authblk}

\usepackage[all]{xy}
\usepackage{amsmath,amssymb,amsfonts,amsthm}
\usepackage{eucal}
\usepackage{verbatim} 
\usepackage{graphicx} 

\usepackage{lineno}
\usepackage{ifthen} 

\usepackage{stmaryrd}

\newtheorem{theorem}{Theorem}[section]

\newtheorem{corollary}[theorem]{Corollary}

\theoremstyle{definition}

\theoremstyle{remark}

\newtheorem{remark}[theorem]{Remark}

\numberwithin{equation}{section}

\begin{document}


\title{On the global dynamics of periodic triangular maps}

\author[1,2]{Rafael Lu\'{i}s\thanks{rafael.luis.madeira@gmail.com}}
\affil[2]{University of Madeira, Funchal, Madeira, Portugal.}
\affil[3]{Center for Mathematical Analysis, Geometry, and Dynamical Systems,
University of Lisbon, Lisbon, Portugal.}


\maketitle

\begin{abstract}
  This paper is an extension of an earlier paper that dealt with global dynamics in autonomous triangular maps. In the current paper, we extend the results on global dynamics of autonomous triangular maps to periodic non-autonomous triangular maps. We show that, under certain conditions,  the orbit of every point in a periodic non-autonomous triangular map converges to a fixed point (respectively, periodic orbit of period $p$) if and only if there is no periodic orbit of prime
  period two (respectively, periodic orbits of prime period greater than $p$). 
\end{abstract}

\textbf{Keys Words:} {Periodic triangular maps, Periodic non-autonomous difference equations, Stability, Global dynamics, Applications}

\textbf{AMS:} {Primary 37C75},  Secondary{ 39A23, 39A30.}

\section{Introduction}
Taking advantage of the structure of triangular maps, Balreira, Elaydi and Lu\'{\i}s \cite{BELTriangular2014}  have been shown global dynamics of higher dimensional autonomous triangular maps. Basically, they consider autonomous continuous triangular maps on $I^k$, where $I$ is a compact interval in the Euclidean space $\mathbb{R}$. It is shown that, under some conditions, the orbit of every point in a triangular map converges to a fixed point if and only if there is no periodic orbits of prime period two. As a consequence of this conclusion, a result on global stability have been investigated. Namely, if there are no
periodic orbits of prime period 2 and the triangular map has a unique fixed point, then the
fixed point is globally asymptotically stable. This last conclusion takes a paramount importance in some applications. Particularly, in population dynamics when we are interested to investigate when local stability implies global stability.

In the current paper, we extend the ideas presents in \cite{BELTriangular2014}, to a periodic non-autonomous triangular difference equations. We show, under certain conditions, global dynamics of fixed points or periodic orbits in the periodic non-autonomous difference equation
\[
X_{n+1}=F_n(X_n),
\]
where $X\in I^k$ and $F_n$ is a periodic sequence of triangular maps. 

The paper is organized as follows: in Section \ref{sec_prel} we introduce the relevant concepts to our analysis. In the next section, we deal with periodic non-autonomous triangular maps and give the results in global dynamics and global stability in such periodic non-autonomous dynamical systems. Finally, in Section \ref{examples} we provide illustrative examples. Particularly, we study a periodic triangular Leslie-Gower competition model, a periodic triangular Ricker type model and a periodic triangular model given by some affine combination. 

It is important  to mention that, in the study of some periodic non-autonomous concrete models, one of the main difficulties encountered is precisely to find analytically the fixed points of the composition. This is precisely the nature of non-autonomous models. That computations are possible only in some limited models and in general, we are not able to find analytically, periodic orbits with period greater than the period of the system. Hence, in this field, is natural to use a software that help us in such computations. In the examples that we provide in this paper we are able to solve analytically the equations. We point out that some known one-dimensional examples can be extended  to a higher dimension as a periodic triangular map, such as the logistic model for instance. However, the computations of the periodic points of the periodic system are possible only by  numerical methods.

Before end this section, we notice that recently, the study of the dynamics of triangular maps has been increase. To cite a few papers, in \cite{manosa2014} Cima, Gassul and Ma\~{n}osa study the dynamics of a planar triangular map of the form $x_{n+1}=f(u_n)+g(u_n)x_n$ and $u_{n+1}=\phi(u_n)$. In \cite{BALIBREA203} the authors study the $\omega-$limit sets of some $3-$dimensional triangular maps on the unit cube.

In the terminology of population dynamics,  triangular models are known as hierarchical models. In the paper \cite{AEKLR2015a} Assas et al. have been study a two-species hierarchical competition model with a strong Allee effect with no immigration (the Allee effect is assumed to be caused by predator saturation). Later on, in \cite{AEKLR2015b} the authors simplified the techniques employed  in the previous paper and provided a proof in global dynamics of the non-hyperbolic cases that were previously conjectured. They have been shown how immigration to one of the species or to both would, drastically, change the dynamics of the system. More precisely, it is shown that if the level of immigration to one or to both species is above a specified level, then there will be no extinction region where both species go to extinction.

\section{Preliminaries}\label{sec_prel}

A map $F: \mathbb{R}^k \rightarrow \mathbb{R}^k$ is said to be \emph{triangular} if it can be written as 
\[
F(x_1, x_2, ..., x_k) = (f_1(x_1), f_2(x_1,x_2), ... , f_k(x_1, x_2, ..., x_k))
\]
For a point $\mathbf{x} \in \mathbb{R}^k$, the orbit of $\mathbf{x}$ under $F$ is given by 
\[
O(\mathbf{x}) = \{F^n(\mathbf{x})|\,n \in \mathbb{Z}^+\},
\] 

\noindent where $\mathbb{Z}^+$ is the set of nonnegative integers. The omega limit set $\omega(\mathbf{x})$ of $\mathbf{x}\in \mathbb{R}^k$ is defined as \[\omega(\mathbf{x}) =\{z \in \mathbb{R}^k |\, F^{n_i}(\mathbf{x}) \rightarrow z,\text{ as }n_i \rightarrow \infty\},\] for some subsequence $n_i$ of $\mathbb{Z}^+$. The set $\omega(\mathbf{x})$ is closed and invariant and it is non-empty if the orbit closure of $\mathbf{x}$ is compact.

The following theorem by W. A. Coppel deals with global dynamics of one-dimensional maps.

\begin{theorem}[W. A. Coppel, 1955~\protect\cite{Coppel1955}]
\label{thm_no_periodic_points} Let $I=[a,b]\subseteq \mathbb{R}$ and $%
f:I\rightarrow I$ be a continuous map. If the equation $f\left( f\left(
x\right) \right) =x$ has no roots, with the possible exception of the roots
of the equation $f\left( x\right) =x\,$, then every orbit under the map $f$
converges to a fixed point.
\end{theorem}

Throughout this paper we will dell with the following notions of stability. Let us consider the map $F: I^k \rightarrow I^k$ of class $C^1$ and let $\mathbf{x}^*$ be a fixed point of the map $F$. We say that:
\begin{enumerate}
	\item $\mathbf{x}^*$ is \textit{stable} if for any $\epsilon>0$, there is $\delta>0$ such that for all $\mathbf{x}_0\in I^k$ with $\|\mathbf{x}_0-\mathbf{x}^*\| < \delta$ we have $\|F(\mathbf{x}_0) - F(\mathbf{x}^*)\|<\epsilon$. Otherwise, $\mathbf{x}^*$ is said to be unstable.
	\item $\mathbf{x}^*$ is \textit{attracting} if there is $\eta>0$ such that $\|\mathbf{x}_0-\mathbf{x}^*\| < \eta$ implies $\lim_{n\to\infty} F^n(\mathbf{x}_0) = \mathbf{x}^*$.
	\item $\mathbf{x}^*$ is \textit{asymptotically stable} if it is both stable and attracting.	
	\item $\mathbf{x}^*$ is globally asymptotically stable if the above definition holds for all $\eta$.
\end{enumerate}

Notice that the above definitions 1. - 3. are local, that is, only need to be satisfied on a neighborhood of $\mathbf{x}^*$.

In order to study the stability in higher dimensional models, we have to look at the eigenvalues of the Jacobian matrix of~$F$ at $\mathbf{x}^*$, denoted by $JF(\mathbf{x}^*)$. Indeed, let us define $\sigma\left(\mathbf{x}^*\right) = \left\{\lambda\in \mathbb{C} |\, \lambda\mbox{ is an eigenvalue of }JF(\mathbf{x}^*) \right\}$. When clear from context, we will simple denote $\sigma\left(\mathbf{x}^*\right)$ as $\sigma$. We have that $\sigma$ decomposes as the subsets
$\sigma^s = \{\lambda \in \mathbb{C} |\,|\lambda| < 1\}$, $\sigma^c = \{\lambda \in \mathbb{C} |\,|\lambda| = 1\}$ and $\sigma^u = \{\lambda \in \mathbb{C} |\,|\lambda| > 1\}$.

For each of the subsets above, the tangent space of $\mathbf{x}^*$ is decomposed into the invariant subspaces $E^s$, $E^c$, $E^u$ corresponding to $\sigma ^s$, $\sigma ^c$, $\sigma ^u$, that is, $\mathbb{R}^k=E^s \oplus E^c \oplus E^u$. It should be noted that some of these subspaces may be trivial subspaces.  

Given an open neighborhood $U$ of $\mathbf{x}^*$, the local stable manifold for $\mathbf{x}^*$ in this neighborhood is defined to be the set

\[
W^s_{loc}(\mathbf{x}^*, U)=\{z \in U |\, \lim \limits_{n\to \infty} F^n(z)=\mathbf{x}^*\}
\]
and the local unstable manifold for $\mathbf{x}^*$ in the neighborhood $U$ is the set
\begin{eqnarray*}
	W^u_{loc}(\mathbf{x}^*, U)&=&\{z \in U|\, \mbox{ there exists a complete negative orbit } \{\mathbf{x}^*_{-n}\} \subset U \\&&\mbox{ such that } \lim \limits_{n\to \infty} \mathbf{x}^*_{-n}=\mathbf{x}^*\}
\end{eqnarray*}

The stable manifold theory \cite{ElaydiDC,Robinson1999}  guarantees the existence of the local stable and unstable manifolds in a suitable open neighborhood $U$ of the fixed point $\mathbf{x}^*$.  In this case, we denote the stable, center, and unstable manifolds of $\mathbf{x}^*$ by $W^s_{loc}(\mathbf{x}^*)$,  $W^c_{loc}(\mathbf{x}^*)$, $W^u_{loc}(\mathbf{x}^*)$, respectively.  Moreover, all these manifolds are invariant.

Taking advantage of the structure of triangular maps, Balreira, Elaydi and Luis \cite{BELTriangular2014} extended Theorem \ref{thm_no_periodic_points} to to $k-$dimensional triangular maps. We now give the main result in \cite{BELTriangular2014}.

\begin{theorem}[Balreira, Elaydi \& Lu\'{\i}s, 2014,~\cite{BELTriangular2014}]\label{main_th_BEL}
Let $F: I^k \rightarrow I^k$ be a continuous $C^1$ triangular map such that each fixed point $\mathbf{x}^*$ is a locally stable fixed point of $F|_{W^c_{loc} (\mathbf{x}^*)}$. Then every orbit in $I^k$ converge to a fixed point if and only if there are no periodic orbits of prime period~two.
\end{theorem}

An important consequence of this theorem is the following corollary in global stability:

\begin{corollary}[\cite{BELTriangular2014}]
Suppose that the triangular map $F$ has a unique fixed point $\mathbf{x}^*$ and $\mathbf{x}^*$ is stable for $F|_{W^c_{loc} (\mathbf{x}^*)}$.  Then $\mathbf{x}^*$ is globally asymptotically stable if and only if $F$ has no periodic orbits of prime period two.
\label{cor_main_th_BEL}
\end{corollary}

Corollary \ref{cor_main_th_BEL}  play a paramount importance in certain models. Particularly, in population dynamics, when we want to show that local stability implies global stability. Notice that in population dynamics, the domain is the nonnegative orthant $\mathbb{R}_+^k$, the extra assumption that all orbits are bounded is needed and the results  remain valid.

The preceding results used the fact that the one-dimensional Sharkovsky's Theorem\footnote{Sharkovsky's Theorem was established in 1964 by Alexander Sharkovsky in Russian \cite{Shark1964}, and later	translated to English by J. Tolosa in 1995 \cite{Shark1995}. It has  played a role of paramount importance  in the dynamics of one-dimensional maps.} holds for $k-$dimensional triangular maps as was shown by Kloeden~\cite{Klo1979}.

\section{Periodic triangular maps}
In this section, we extend the results presented in the previous section to periodic triangular non-autonomous models. We begin by presenting the principal concepts and notations.

Let us consider a triangular non-autonomous discrete difference equation given by
\begin{equation}
X_{n+1}=F_n(X_n),\label{TDDS}
\end{equation}
where $X_n\in I^k$ and the sequence of triangular maps $F_n=(f_{1,n},f_{2,n},\ldots,f_{k,n})$ are of class $C^1$ of the form
\begin{eqnarray*}
F_n(X_n)&=&F_n(x_{1,n},x_{2,n},\ldots,x_{k,n})\\
&=&(f_{1,n}(x_{1,n}),f_{2,n}(x_{1,n},x_{2,n}),\ldots,f_{k,n}(x_{1,n},x_{2,n},\ldots,x_{k,n})).
\end{eqnarray*}

The orbit of an initial point $X_0=(x_{1,0},x_{2,0},\ldots,x_{k,0})$ in the System (\ref{TDDS}) is given by $O(X_0)=\{X_0,X_1,X_2,\ldots\}$ where
\begin{eqnarray*}
X_1&=&(x_{1,1},x_{2,1},\ldots,x_{k,1})=F_0((x_{1,0},x_{2,0},\ldots,x_{k,0}))\\
&=&(f_{1,0}(x_{1,0}),f_{2,0}(x_{1,0},x_{2,0}),\ldots,f_{k,0}(x_{1,0},x_{2,0},\ldots,x_{k,0})),
\end{eqnarray*}

\begin{eqnarray*}
	X_2&=&(x_{1,2},x_{2,2},\ldots,x_{k,2})=F_1\circ F_0((x_{1,0},x_{2,0},\ldots,x_{k,0}))\\
	&=&(f_{1,1}\left[f_{1,0}(x_{1,0})\right],f_{2,1}\left[f_{1,0}(x_{1,0}),f_{2,0}(x_{1,0},x_{2,0})\right],\ldots,\\
	&&f_{k,1}\left[f_{1,0}(x_{1,0}),f_{2,0}(x_{1,0},x_{2,0}),\ldots,f_{k,0}(x_{1,0},x_{2,0},\ldots,x_{k,0})\right]),
\end{eqnarray*}

\begin{eqnarray*}
	X_3&=&(f_{1,2}\{f_{1,1}\left[f_{1,0}(x_{1,0})\right]\},\ldots,\\
	&&f_{k,2}\left\{f_{k,1}\left[f_{1,0}(x_{1,0}),f_{2,0}(x_{1,0},x_{2,0}),\ldots,f_{k,0}(x_{1,0},x_{2,0},\ldots,x_{k,0})\right]\right\}),
\end{eqnarray*}
etc..

Define the triangular composition map as follows
\begin{eqnarray*}
\Phi_n(X):&=&\Phi_n(x_1,\ldots,x_k)=F_{n-1}\circ \ldots\circ F_1\circ F_0(X)\\&=&(\phi_{1,n}(x_1),\phi_{2,n}(x_1,x_2),\ldots,\phi_{k,n}(x_1,\ldots,x_k))
\end{eqnarray*}
where $\phi_{j,n}(x_1,\ldots, x_j)$, $j=1,2,\ldots, k$, is the composition obtained in each coordinate as illustrated before. Notice that we are assuming that the composition is well defined in some subset of $I^k$. Hence, the dynamics of the system is completely determined by the triangular composition map $\Phi_n(X)$.

Now let us consider the periodicity of Equation (\ref{TDDS}) by considering  $p_i$ to be the minimal period of the individual maps $f_{j,i}$, i.e., for each $j=1,2,\ldots, k$ we have $f_{j,n+p_i}=f_{j,n}$ for all $n=0,1,2,\ldots$. Under this scenario we have $F_{n+p}=F_n$, for all $n=0,1,2,\ldots$, where $$p=lcm\{p_1,p_2,\ldots,p_k\}$$ is minimal. Hence, Equation (\ref{TDDS}) is $p-$periodic and its dynamics is completely determined by the composition triangular map $\Phi_p$. 

Notice that we can have the following three scenarios in the dynamics of the equation (\ref{TDDS}):
\begin{enumerate}
	\item \textbf{Common fixed point.} If $x^*$ is a common fixed point of the sequence of maps $F_n$ (and obviously a fixed point of the $p-$periodic equation (\ref{TDDS}), then $X^*$ is also a fixed point of the composition triangular map $\Phi_{q,i}$;
	\item \textbf{Cycle.} If there exists $q\in\mathbb{Z}^+$ such that $mq=p$ for some $m=2,3\ldots$, then the $p-$periodic Equation (\ref{TDDS}) has a $q-$periodic cycle which is a $mq-$periodic cycle of the triangular composition map $\Phi_{p,i}$. This cycle is denoted by $$\mathcal{C}_q=\{\overline{X}_{i},\overline{X}_{i+1},\ldots,\overline{X}_{i+q-1}\}.$$ Notice that $\overline{X}_i$, $i=0,1,\ldots,q-1$  is fixed point of the following compositions: 
	\begin{eqnarray*}
	\Phi_{q_1,i}(\overline{X}_i)&=&F_{i+q-1}\circ\ldots\circ F_{i+1}\circ F_i(\overline{X}_i),\\
	\Phi_{q_2,i}(\overline{X}_i)&=&F_{i+2q-1}\circ\ldots\circ F_{i+q+1}\circ F_{i+q}(\overline{X}_i),\\
	&&\hspace{2cm}\ldots\\
	\Phi_{q_m,i}(\overline{X}_i)&=&F_{i+mq-1}\circ\ldots\circ F_{i+(m-1)q+1}\circ F_{i+(m-1)q}(\overline{X}_i).
	\end{eqnarray*}
	\item \textbf{Geometric cycle.} If the triangular composition map $\Phi_{p,i}$ has a fixed point which is not in the conditions described above, then it is a $p-$periodic cycle of the $p-$periodic Equation (\ref{TDDS}). In this case this cycle is called as a geometric cycle and denoted by  $$\mathcal{C}_p=\{\overline{X}_{i},\overline{X}_{i+1},\ldots,\overline{X}_{i+p-1}\}.$$ Notice that $\overline{X}_i$, $i=0,1,2\ldots,p-1$, is a fixed point of the following compositions
	\begin{eqnarray}
	\Phi_{p,i}(\overline{X}_i)=F_{i+p-1}\circ\ldots\circ F_{i+1}\circ F_{i}(\overline{X}_i).
	\end{eqnarray}
	Here, when $i=0$ we denote $\Phi_{p,0}$ by $\Phi_p$ as described above.
\end{enumerate}

Notice that, since we are dealing with global dynamics, in the last case it is enough to study the qualitative properties of the fixed point $\overline{X}_0$ of the full composition $\Phi_p$ since it determines completely the dynamics of the $p-$periodic equation (\ref{TDDS}). The other cases follow in a similar fashion.

We point out that the non-autonomous periodic difference equation (\ref{TDDS})
does not generate a discrete (semi)dynamical system \cite{ESUAE} as it may
not satisfy the (semi)group property. One of the most effective ways of
converting the non-autonomous difference equation (\ref{TDDS}) into a
genuine discrete (semi)dynamical system is the construction of the
associated skew-product system as described in a series of papers by Elaydi
and Sacker \cite{ElSa2005, ElSa2005a, ESUAE, ES2006}. It is noteworthy to
mention that this idea was originally used to study non-autonomous
differential equations by Sacker and Sell \cite{SS1977}.

We are now in position to extend the results of the previous section to the $p-$periodic equation (\ref{TDDS}). Recall that we are assuming that $F_{n+p}=F_n$ for all $n=0,1,2,\dots$, in the described conditions.

 \begin{theorem}\label{main_th_FP}
 	Let $\mathbf{X}^*$ to be a common fixed point of the periodic sequence of continuous $C^1$ triangular maps $F_n$ with $F_n: I^k \rightarrow I^k$ such that each fixed point $\mathbf{X}^*$ is a locally stable fixed point of $F_n|_{W^c_{loc} (\mathbf{x}^*)}$. Then every orbit in $I^k$ of the $p-$periodic Equation (\ref{TDDS}) converge to a  fixed point if and only if there are no periodic orbits of prime period $2$.
 \end{theorem}
\begin{proof}
	The proof follows the techniques employed in the proof of Theorem \ref{main_th_BEL} presented in \cite{BELTriangular2014} by considering that in each fiber $\mathbb{F}_i$, $i=0,1,2\ldots,p-1$ (here each fiber is a simple copy of our space $I^k$), there exists a subsequence $n_j$, $j=1,2,\ldots, k$ where the dynamics holds in each coordinate map $f_{j,i}$.
\end{proof}

In the case that Equation (\ref{TDDS}) possesses a cycle $\mathcal{C}_q$, we follow a similar idea as before, considering that in each fiber $\mathbb{F}_{i+l}$, $i=0,1,\ldots, q-1$ and $l=0,q,2q,\ldots,mq-1$, there exists a subsequence $n_j$, $j=1,2,\ldots, k$, where the dynamics holds in each coordinate map $f_{j,i}$.

\begin{theorem}\label{main_th_ICy}
	Let $mq=p$ and for $j=0,1,\ldots,m$ assume that $\Phi_{qj,i}: I^k \rightarrow I^k$ is a continuous $C^1$ triangular map such that each fixed point $\overline{X}_i$, $i=0,1,\ldots, q-1$, of $\Phi_{qj,i}$ is a locally stable fixed point of $\Phi_{qj,i}|_{W^c_{loc} (\overline{X}_i)}$.	 Then every orbit in $I^k$ of the $p-$periodic Equation (\ref{TDDS}) converge to a  $q-$periodic cycle $\mathcal{C}_q=\{\overline{X}_i,\overline{X}_{i+1},\ldots,\overline{X}_{i+q-1}\}$ if and only if there are no periodic orbits of prime period greater than $q$.
\end{theorem}

In the case of a geometric $p-$periodic cycle $\mathcal{C}_p$ we have the following result.

\begin{theorem}\label{main_th_PTM}
	Let $\Phi_{p,i}: I^k \rightarrow I^k$ be a continuous $C^1$ triangular map such that each fixed point $\overline{X}_i$, $i=0,1,\ldots, p-1$ of $\Phi_{p,i}$ is a locally stable fixed point of $\Phi_{p,i}|_{W^c_{loc} (\overline{X}_i)}$. Then every orbit in $I^k$ of the $p-$periodic Equation (\ref{TDDS}) converge to a geometric $p-$periodic cycle $\mathcal{C}_p=\{\overline{X}_i,\overline{X}_{i+1},\ldots,\overline{X}_{i+p-1}\}$ if and only if there are no periodic orbits of prime period greater that $p$.
\end{theorem}
\begin{proof}
	Let us consider the autonomous difference equation $X_{n+1}=\Phi_{p,i}(X_n)$, for all $n=0,1,2,\ldots$ and $i=0,1,\ldots, p-1$. Since $\Phi_{p,i}$ is a continuous $C^1$ triangular map whose fixed point $\overline{X}_i$ is locally stable fixed point of $\Phi_{p,i}|_{W^c_{loc} (\overline{X}_i)}$, it follows by Theorem \ref{main_th_BEL} that in each fiber $\mathbb{F}_i$, $i=0,1,\ldots, p-1$, every orbit  of $X_{n+1}=\Phi_{p,i}(X_n)$ converges to a fixed point $\overline{X}_i$ if and only if there are no periodic orbits of prime period $2$. 
	
	Since the fixed point $\overline{X}_i$ of $\Phi_{p,i}$ is a periodic point with period $p$ of  the $p-$periodic equation $X_{n+1}=F_n(X_n)$, i.e., the fixed point $\overline{X}_i$ of $\Phi_{p,i}$ generates a geometric cycle of the form $\mathcal{C}_p=\{\overline{X}_i,\overline{X}_{i+1},\ldots,\overline{X}_{i+p-1}\}$, then the above conclusion is equivalent to say that every orbit in $\mathbb{F}_i$, $i=0,1,\ldots, p-1$ of $X_{n+1}=F_n(X_n)$,  converge to a $p-$periodic point if and only if there are no periodic orbits of prime period greater than $p$. 
\end{proof}

\begin{remark}\label{rGS}
	\begin{enumerate}
		It is clear that: 
		\item If we assume uniqueness of $\mathbf{X}^*$ in Theorem \ref{main_th_FP}, then $\mathbf{X}^*$ is globally asymptotically stable if and only if $X_{n+1}=F_n(X_n)$ has no periodic orbits of prime period two.
		\item If we assume uniqueness of $\mathcal{C}_q$ in Theorem \ref{main_th_ICy}, then $\mathcal{C}_q$ is globally asymptotically stable if and only if the map $\Phi_{qj,i}$, $i=0,1,\ldots, q-1$ and $j=1,2,\ldots,m$, has no periodic orbits of prime period two.
		\item If we assume uniqueness of $\mathcal{C}_p$ in Theorem \ref{main_th_PTM}, then unique $p-$periodic orbit of $x_{n+1}=F_n(X_n)$  is globally asymptotically stable if and only if $\Phi_{p,i}$, $i=0,1,\ldots, p-1$, has no periodic orbits of prime period two.
	\end{enumerate}
\end{remark}

\section{Illustrative examples}\label{examples}
In this section we give some illustrative examples in order to show  the ideas presented in the previous section and the effectiveness of our tools. 
\subsection{Triangular periodic Leslie-Gower competition model}\label{PTLG}
Let us consider the non-autonomous triangular map $F_n: \mathbb{R}_+^2 \rightarrow \mathbb{R}_+^2$ given by
	\begin{equation}\label{eq:Leslie_Gower}
	F_n(x, y)=\left(\frac{\displaystyle \mu K_nx}{\displaystyle K_n+(\mu - 1)x}, \frac{\displaystyle \alpha L_n y}{\displaystyle L_n+(\alpha-1)y+\beta x}\right),
	\end{equation}
	
	\noindent where $\mu, \alpha >1$, $0<\beta<1$ and $K_n, L_n >0$. In order to have periodicity of the equation
	\begin{equation}
	X_{n+1}=F_n(X_n)\label{PEQ}
	\end{equation} 
	let us assume that $K_{n+2}=K_n$ and $L_{n+2}=L_n$, for all $n=0,1,2,\ldots$. Hence, the non-autonomous equation (\ref{PEQ}) is $2-$periodic.
	
	Notice that the orbits of $F_n$ are bounded since $\frac{\mu K_n x}{K_n+(\mu-1)x}<\frac{\mu K_n x}{(\mu-1)x}=\frac{\mu K_n}{\mu-1}$ and $\frac{ \alpha L_n y}{L_n+(\alpha-1)y+\beta x}<\frac{ \alpha L_n y}{ (\alpha-1)y}=\frac{\alpha L_n}{\alpha-1}$, for all $n=0,1,2,\ldots$.
	
	It is a straightforward computation to see that:
	\begin{enumerate}
		\item $O=(0,0)$ is a fixed point of Equation (\ref{PEQ});
		\item $\mathcal{E}_x=\left\{(\overline{x}_0,0),(\overline{x}_1,0)\right\}=\left\{\left(\frac{K_0 K_1 (\mu +1)}{K_0 \mu +K_1},0\right),\left(\frac{K_0 K_1 (\mu +1)}{K_1 \mu +K_0},0\right)\right\}$ is an exclusion $2-$periodic cycle Equation (\ref{PEQ}) on the $x-$axis;
		\item $\mathcal{E}_y=\left\{(0,\overline{y}_1),(0,\overline{y}_2)\right\}=\left\{\left(0,\frac{(\alpha +1) L_0 L_1}{\alpha  L_0+L_1}\right),\left(0,\frac{(\alpha +1) L_0 L_1}{\alpha  L_1+L_0}\right)\right\}$ is an exclusion $2-$periodic cycle of Equation (\ref{PEQ}) on the $y-$axis;
		\item $\mathcal{C}_2=\left\{\left(\overline{x}_0,\overline{Y}_0\right),\left(\overline{x}_1,\overline{Y}_1\right)\right\}$ is a coexistence $2-$periodic cycle of Equation (\ref{PEQ}).
	\end{enumerate}
where $\overline{Y}_0$ and $\overline{Y}_1$ are given by 
\[
\overline{Y}_0=\frac{\left(\alpha ^2-1\right) K_1^2 \mu  L_0 L_1-A K_0^2 -B K_1 K_0 }{(\alpha -1) \left(K_0 \mu +K_1\right) \left(K_0 \left(\beta  K_1 (\mu +1)+\alpha  L_0+L_1\right)+K_1 \mu  \left(\alpha  L_0+L_1\right)\right)}
\]
and 
\[
\overline{Y}_1=\frac{\left(\alpha ^2-1\right) K_1^2 \mu  L_0 L_1-A K_0^2 -B K_1 K_0 }{(\alpha -1) \left(K_1 \mu +K_0\right) \left(K_0 \left(\beta  K_1 (\mu +1)+\mu  \left(\alpha  L_1+L_0\right)\right)+K_1 \left(\alpha  L_1+L_0\right)\right)}
\]
with 
$$A=\beta ^2 K_1^2 (\mu +1)^2+\beta  K_1 (\mu +1) \left(\mu  L_0+L_1\right)-\left(\alpha ^2-1\right) \mu  L_0 L_1$$
and 
$$B=\beta  K_1 (\mu +1) \left(\mu  L_1+L_0\right)-\left(\alpha ^2-1\right) \left(\mu ^2+1\right) L_0 L_1$$

In order to guarantee that $\mathcal{C}_2$ belongs to the first quadrant we assume that
\begin{equation}
\left(\alpha ^2-1\right) K_1^2 \mu  L_0 L_1>A K_0^2 +B K_1 K_0 
\label{cpc}
\end{equation}
Now we will study the properties of the fixed points of the triangular composition map $\Phi_2=F_1\circ F_0$. (As observed before the map $\Phi_{2,1}=F_0\circ F_1$ leads to the same conclusions and will be omitted).

From the above conclusions we see that the map $\Phi_2$ has the fixed points $O=(0,0)$, $\varepsilon_x=(\overline{x}_0,0)$, $\varepsilon_y=(0,\overline{y}_0)$ and $\mathcal{C}^*=(\overline{x}_0,\overline{Y}_0)$.

The Jacobian matrix of $\Phi_2$, $J\Phi_2$, evaluated at the fixed points is 
\[
J\Phi_2(O)=\left(
\begin{array}{cc}
\mu ^2 & 0 \\
0 & \alpha ^2 \\
\end{array}
\right),
\]
\[
J\Phi_2(\varepsilon_x)=\left(
\begin{array}{cc}
\frac{1}{\mu ^2} & 0 \\
0 & \frac{\alpha ^2 \left(\mu  K_0+K_1\right) \left(K_0+\mu  K_1\right) L_0 L_1}{\left(K_1 L_0+K_0 \left(\beta  (\mu +1) K_1+\mu  L_0\right)\right) \left(\mu  K_1 L_1+K_0 \left(\beta  (\mu +1) K_1+L_1\right)\right)} \\
\end{array}
\right),
\]
\[
J\Phi_2(\varepsilon_y)=\left(
\begin{array}{cc}
\mu ^2 & 0 \\
-\frac{(\alpha +1) \beta  \left(\alpha  \mu  L_0^2+\alpha  (\alpha  \mu +1) L_1 L_0+L_1^2\right)}{\alpha ^2 \left(\alpha  L_0+L_1\right){}^2} & \frac{1}{\alpha ^2} \\
\end{array}
\right)
\]
and 
\[
J\Phi_2(\mathcal{C}^*)=\left(
\begin{array}{cc}
\frac{1}{\mu ^2} & 0 \\
------ & \frac{\left(K_1 L_0+K_0 \left(\beta  (\mu +1) K_1+\mu  L_0\right)\right) \left(\mu  K_1 L_1+K_0 \left(\beta  (\mu +1) K_1+L_1\right)\right)}{\alpha ^2 \left(\mu  K_0+K_1\right) \left(K_0+\mu  K_1\right) L_0 L_1} \\
\end{array}
\right)
\]
The origin is a repeller fixed point since the spectrum of the matrix $J\Phi_2(0)$ is outside  the unit disk.  Since $\frac{1}{\alpha}<1<\mu$, it follows that the fixed point $\varepsilon_y$ is a saddle. 

After some algebraic manipulations, it follows from Relation (\ref{cpc}) that 
\[
\frac{\alpha ^2 \left(\mu  K_0+K_1\right) \left(K_0+\mu  K_1\right) L_0 L_1}{\left(K_1 L_0+K_0 \left(\beta  (\mu +1) K_1+\mu  L_0\right)\right) \left(\mu  K_1 L_1+K_0 \left(\beta  (\mu +1) K_1+L_1\right)\right)}>1.
\]
From this we conclude that the exclusion fixed point $\varepsilon_x$ is a saddle since $\frac{1}{\mu ^2}<1$. Moreover, the spectrum of the coexistence fixed point $\mathcal{C}^*$ lie inside the unit disk.

We can also see that the map $\Phi_2(x,y)=(\Phi_{1,2}(x),\Phi_{2,2}(x,y))$ has a periodic orbit of prime period $2$ whenever 
	\[
	\left\{ 
	\begin{array}{l}
	f_{1,1}\circ f_{1,0}\circ f_{1,1}\circ f_{1,0}(x)=x \\ 
	f_{2,1}\left\{f_{1,0}(x),f_{2,0}\left[f_{1,1}(x),f_{2,1}(f_{1,0}(x),f_{2,0}(x,y))\right]\right\}=y
	\end{array}
	\right. ,
	\]
	where
	\[
	f_{1,n}(x)= \frac{\displaystyle \mu K_nx}{\displaystyle K_n+(\mu - 1)x}\text{ and }f_{2,n}(x,y)=\frac{\displaystyle \alpha L_n y}{\displaystyle L_n+(\alpha-1)y+\beta x}.
	\]
Using a software such as Mathematica or Maple, one can verify that the solutions of the previous system are the points $O$, $\varepsilon_x$, $\varepsilon_y$ and $\mathcal{C}^*$, precisely the fixed points of the map $\Phi_2$. These computations  can be done analytically, however it involves long expressions. So, we take advantage using of the software since it is able to do it algebraically. 

From this we conclude that there are no periodic orbits of prime period $2$ under the action of the composition operator $\Phi_2$. Since the hypotheses of Theorem~\ref{main_th_PTM} are satisfied, every orbit converges  to a fixed point of $\Phi_2$ which is a geometric $2-$periodic cycle of Equation (\ref{PEQ}).

Moreover, since $\mathcal{C}^*$ is the unique fixed point of $\Phi_2$ in the interior of the first quadrant and all the eigenvalues lies in the interior of the unique disk, it follows by remark \ref{rGS} that $\mathcal{C}_2$ is a globally asymptotically stable $2-$periodic cycle of Equation (\ref{PEQ}) with respect to the interior of the first quadrant.
\subsection{Triangular periodic Logistic-Type Map}
Consider the non-autonomous triangular logistic-type map defined on $I^2=[0,1]^2$ given by
\begin{equation}
F_n(x,y)=(\mu_n x(1-x),\nu_n y(1-y)x),
\end{equation}
where $\mu_i>0$ and $\nu_i>0$, for all $i=0,1,2,\ldots$.

Each individual map $F_i$, $i=0,1,2,\ldots$, has the following fixed points
\[
E_0=(0,0),~E_{1,i}=\left(\frac{\mu_i-1}{\mu_i},0\right)\text{ and }E_{2,i}=\left(\frac{\mu_i-1}{\mu_i},\frac{\mu_i+(1-\mu_i)\nu_i}{(1-\mu_i)\nu_i}\right).
\]
In Table \ref{tifpoints}  is presented the regions of local stability, in the parameter space,  of the individual fixed points of the maps $F_i$.

\begin{table}[h!]
	\begin{center}
		\begin{tabular}{c|c}
			Individual fixed point & Region of local stability \\ \hline\hline
			$E_{0}$ & $\mu _{i}\leq 1$ \\
			$E_{1,i}$ & $1<\mu _{i}\leq 3~~\wedge~~\dfrac{\mu_i-1}{\mu_i}\nu_i\leq 1 $ \\
			$E_{2,i}$ & $1<\mu _{i}\leq 3~~\wedge~~1<\dfrac{\mu_i-1}{\mu_i}\nu_i\leq 3$ \\
			\end{tabular}
		\caption{Regions, in the parameter space, of local stability of the fixed points $E_0$, $E_{1,i}$ and $E_{2,i}$, $i=0,1,2,\ldots$ of each individual the triangular logistic map.}
		\label{tifpoints}
	\end{center}
\end{table}

Notice that the orbits of $F_n$ are bounded since $\mu_nx(1-x)\leq\frac{\mu_n}{4}$ and $\nu_ny(1-y)x\leq\frac{\mu_i\nu_i}{16}$, for all $n=0,1,2,\ldots$.

Now, let us consider, for simplicity, the periodicity of the equation by taking $\mu_{n+2}=\mu_n$ and $\nu_{n+2}=\nu_n$, for all $n=0,1,2,\ldots$. This implies that the non-autonomous equation
\begin{equation}
X_{n+1}=F_n(X_n)\label{PLE}
\end{equation}
  is $2-$periodic since the sequence of maps $F_n$ is $2-$periodic. 
  
  In order to study the dynamics of Equation (\ref{PLE}) we will consider the composition map $\Phi_2=F_1\circ F_0$. It is a straightforward computation to see that, the fixed points of $\Phi_2$ are:
	\begin{enumerate}
		\item $\mathcal{E}_0=(0,0)$;
		\item $\mathcal{E}_1=(x^*,0)= \left(\dfrac{2}{3}-\dfrac{\sqrt[3]{2^2} \sqrt[3]{\Delta_1+3 \sqrt{3} \sqrt{\Delta_2}}}{6\mu _0^2 \mu _1}-\dfrac{\sqrt[3]{2} \left(\mu _0-3\right) \mu _0 \mu _1}{3\sqrt[3]{\Delta_1+3 \sqrt{3} \sqrt{\Delta_2}}},0\right)$;
		\item $\mathcal{E}_2=(x^*,y^*)$ 
	\end{enumerate}
where
\[
\Delta_1=2 \mu _1^3 \mu _0^6-9 \mu _1^3 \mu _0^5+27 \mu _1^2 \mu _0^4,
\]
\[
\Delta_2=\mu _0^8 \mu _1^4 \left(\left(4-\mu _1\right) \mu _1 \mu _0^2-2 \mu _1 \left(9-2 \mu _1\right) \mu _0+27\right),
\]
and
\[
y^*=\Psi(\mu_0,\mu_1,\nu_0,\nu_1):=\Psi.
\]
Since $\Psi$ involves a long expression, we omit it  here. We point out that it can be obtained explicitly by using a software such as Mathematica or Maple.  Notice that $x^*$ is real whenever $\Delta_2\geq 0$, i.e., whenever $\left(4-\mu _1\right) \mu _1 \mu _0^2-2 \mu _1 \left(9-2 \mu _1\right) \mu _0+27\geq 0$ (see Figure \ref{fig:Log_real}).
\begin{figure}
\centering
\includegraphics[scale=0.4]{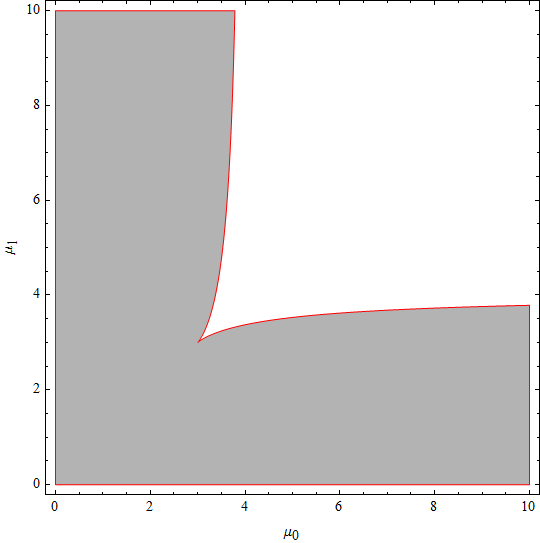}
\caption{Part of the region, in the parameter space $\mu_0O\mu_1$, where where the fixed point $x^*$ of $\Phi_2$ is real, i.e., the region where the following inequality is verified $\left(4-\mu _1\right) \mu _1 \mu _0^2-2 \mu _1 \left(9-2 \mu _1\right) \mu _0+27\geq 0$.}
\label{fig:Log_real}
\end{figure}

\begin{figure}[t]
\centering
\includegraphics[scale=0.4]{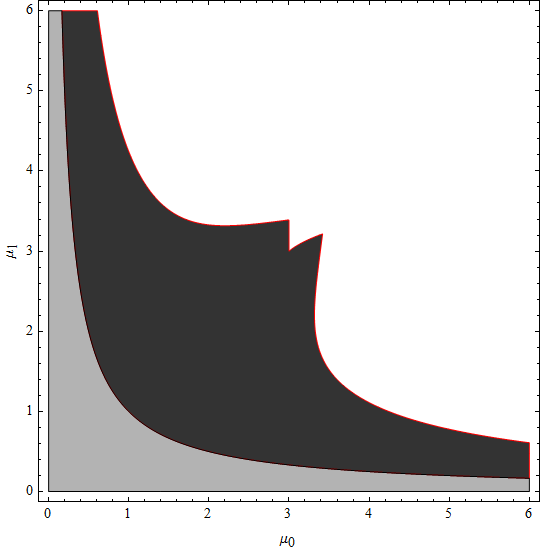}
\caption{Region of local stability, in the parameter space $\mu_0O\mu_1$,  of the fixed points of the $2-$periodic triangular logistic model. In the gray region the fixed point $\mathcal{E}_0$ is locally asymptotically stable while in the black region the fixed point $\mathcal{E}_1$ is locally asymptotically stable. In this second fixed point we consider $\nu_0=\nu_1=0.5$.}
\label{fig:Logistic}
\end{figure}

Now, the spectrum of the Jacobian matrices of $\Phi_2$ evaluated at the fixed points is given by $\sigma_{\mathcal{E}_0}=\{0,\mu_0\mu_1\}$, 
\[
\sigma_{\mathcal{E}_1}=\{\mu_0\mu_1(1+2x^*)(1-2\mu_0x^*(x^*-1)),\mu_0(x^*)^2(1-x^*)\nu_0\nu_1\}
\] 
and
\[
\sigma_{\mathcal{E}_2}=\{\mu_0\mu_1(1+2x^*)(1-2\mu_0x^*(x^*-1)),
\mu_0(x^*)^2(1-x^*)\nu_0\nu_1(1-2\Psi)(1+2\Psi(\Psi-1)x^*\nu_0)
\}.
\]

By considering parameters values as shown in Table~\ref{tpfpoints}, we know that each fixed point $\mathcal{E}_i$, $i=0,1,2$ is locally asymptotically stable. In Figure \ref{fig:Logistic} is presented the region of local stability, in the parameter space $\mu_0O\mu_1$, of the fixed points $\mathcal{E}_0$ and $\mathcal{E}_1$.

Let us now look how Theorem~\ref{main_th_PTM} establishes the global dynamics of the fixed points $\mathcal{E}_i$, $i=0,1,2$. Indeed, in order to apply Theorem~\ref{main_th_PTM}, it suffices to show that there are no periodic points, i.e., the solutions of the equation
\[
\Phi_2\circ\Phi_2(x,y)=(x,y),
\]
or equivalently
\[
F_1\circ F_0\circ F_1\circ F_0(x,y)=(x,y),
\]
are exactly the fixed points $\mathcal{E}_i$, $i=0,1,2$, when the parameters are restricted to the region established in Table \ref{tpfpoints}. 
\begin{table}[h!]
	\begin{center}
		\begin{tabular}{c|c}
			Fixed point & Region of local stability \\ \hline\hline
			$\mathcal{E}_0$ & $\mu_0\mu _{1}< 1$ \\[10pt]
			$\mathcal{E}_1$ &  $\mu_0\mu_1(1+2x^*)|1-2\mu_0x^*(x^*-1)|< 1~~\wedge~~\mu_0(x^*)^2|1-x^*|\nu_0\nu_1< 1$ \\[10pt]
			$\mathcal{E}_2$ &\begin{tabular}{c}
							$\mu_0\mu_1(1+2x^*)|1-2\mu_0x^*(x^*-1)|< 1~~\wedge$\\
							$\mu_0(x^*)^2|1-x^*|\nu_0\nu_1|(1-2\Psi)(1+2\Psi(\Psi-1)x^*\nu_0)|< 1$
							\end{tabular} \\
			\end{tabular}
		\caption{Regions, in the parameter space, of local stability of the fixed points $\mathcal{E}_i$, $i=0,1,2$ of the triangular logistic map $\Phi_2$.}
		\label{tpfpoints}
	\end{center}
\end{table}
Unfortunately, we are not able to solve explicitly the previous equation. At this stage the computations are able only by numerical methods. Using a software such as Mathematica or Maple one can see that $(0,0)$, $(x^*,0)$ and $(x^*,\Psi)$ are solutions of the equation. There are however solutions of some polynomial equation. Numerical computations shows that the solutions of this polynomial equation are not real under the conditions of local stability expressed in Table \ref{tpfpoints}.

Hence, if one assumes that this polynomial has no real roots whenever the parameters belongs to the stability region defined in Table \ref{tpfpoints}, then the map $\Phi_2$ has no periodic points of minimal period $2$. Moreover, all the fixed points $\mathcal{E}_i$, $i=0,1,2$ are hyperbolic and the hypotheses of Theorem \ref{main_th_PTM} are verified. Hence, we conclude that every orbit converges to a fixed point under the action of the composition $\Phi_2$ which is a periodic point of period $2$ of the non-autonomous equation $X_{n+1}=F_n(X_n)$. Moreover, since $\mathcal{E}_2$ is the unique fixed point of $\Phi_2$ in the interior of $I^2$, it follows that the non-autonomous equation $X_{n+1}=F_n(X_n)$ has a globally asymptotically stable $2-$periodic orbit in the interior of $I^2$.

\subsection{Triangular periodic exponential type model}

Let us consider the non-autonomous difference equation 
\begin{equation}
X_{n+1}=F_n(X_n)\label{NADE}
\end{equation} 
where the sequence of triangular maps $F_n: \mathbb{R}_+^2 \rightarrow \mathbb{R}_+^2$ are given by
\begin{equation}\label{eq:ricker}
F_n(x, y)=\left(xe^{r_n(1-x)},ye^{s_n(1-y-\mu x)}\right),
\end{equation}
where $0<r_n\leq 2$ and $0<s_n\leq 2$, for all $n$. In order to have periodicity let us assume that $r_{n+3}=r_n$ and $s_{n+2}=s_n$, for all $n=0,1,2,\ldots$. Hence, the non-autonomous equation (\ref{NADE}) is $6-$periodic.
Notice that all  orbits in (\ref{NADE}) are bounded since $xe^{r_n(1-x)}\leq \frac{1}{r_n}e^{r_n-1}$ and $ye^{s_n(1-y-\mu x)}\leq \frac{1}{s_n}e^{s_n-1}$, for all $n$.

It is a straightforward computation to see that $O=(0,0)$, $\varepsilon_x=(1,0)$, $\varepsilon_y=(0,1)$ and  $\mathcal{C}^*=(1,1-\mu)$ are common fixed points of all the maps $F_n$. Moreover, $O$ is a repeller fixed point,  $\varepsilon_x=(1,0)$ and  $\varepsilon_y=(0,1)$ are saddle and the eigenvalues of the Jacobian matrix $JF_i$ evaluated at the fixed point $\mathcal{C}^*$ lie inside the unit disk when
\[
0<s_i(1-\mu)\leq 2, ~~i=0,1.
\]
Moreover, computations shows that there are no $2-$periodic cycles in Equation (\ref{NADE}) since the solutions of $F_{i+1}\circ F_i(x,y)=(x,y)$, $i=0,1,\ldots,5$, are the fixed points of each individual map $F_i$. Consequently, by Theorem \ref{main_th_FP} every orbit in Equation (\ref{NADE}) converges to a fixed point. Moreover, since  $\mathcal{C}^*$ is the unique fixed point in the interior of the first quadrant, it follows by remark \ref{rGS} that $\mathcal{C}^*$ is globally asymptotically stable fixed point with respect to the interior of the first quadrant.

A suitable generalization of this model leads to the same conclusions. For instance, if we consider the sequence of triangular maps as follow
\[
F_n(x_1,\ldots,x_k)=(x_1e^{r_{1,n}(1-x_1)},x_2e^{r_{2,n}(1-x_2-\mu_1x_1)},\ldots,x_ke^{r_{k,n}(1-x_k-\sum_{j=1}^{k-1}\mu_jx_j)})
\]
where for each $j=1,2,\ldots,k$ we have $r_{j,n+p_j}=r_{j,n}$ such that $0<r_{j,n}\leq 2$, then the $p-$periodic difference equation $X_{n+1}=F_n(n)$, with $p=lcm(p_1,\ldots,p_k)$, has a globally asymptotically stable fixed point  of the form
\begin{eqnarray*}
\mathcal{C}^*&=&\left(1,1-\mu_1,\prod_{i=1}^{2}(1-\mu_i),\ldots,\prod_{i=1}^{k-1}(1-\mu_i)\right)\\
\end{eqnarray*}
with respect to the interior of the first orthant of $\mathbb{R}^k_+$ whenever 
\begin{eqnarray*}
0&<&r_{1,n_1}\leq 2,~~\text{ for all }n_1=0,1,2,\ldots,p_1-1
\end{eqnarray*}
and 
\begin{eqnarray*}
0&<&r_{j,n_j}\prod_{i=1}^{j-1}(1-\mu_i)\leq 2,
\end{eqnarray*}
for all $n_j=0,1,2,\ldots,p_j-1$ and $j=2,3,\ldots,k$.

\section*{Acknowledgements}
This work was partially supported by FCT/Portugal through project PEst-OE/EEI/LA0009/2013

\end{document}